\documentclass{amsart}

\usepackage{graphicx}
\usepackage[latin1]{inputenc}

\newtheorem{thm}{Theorem}[section]
\newtheorem{cor}[thm]{Corollary}
\newtheorem{lem}[thm]{Lemma}

\newtheorem{exa}[thm]{Example}
\theoremstyle{definition}
\newtheorem{dfn}[thm]{Definition}
\theoremstyle{remark}
\newtheorem{rem}[thm]{Remark}
\numberwithin{equation}{section}

\begin{document}

\title[]{The group $G_{n}^{2}$ and Invariants of Free Knots Valued in Free Groups}%
\author{S.Kim, V.O.Manturov}
\address{V.O.Manturov, Bauman Moscow State Technical University,}
\email{vomanturov@yandex.ru}

\address{S.Kim, Bauman Moscow State Technical University,}%
\email{ksj19891120@gmail.com}%


\maketitle

\begin{abstract}
In the present paper, we define an invariant of free links valued in a free product of some copies of $\mathbb{Z}_{2}$. In \cite{Ma2} the second named author constructed a connection between classical braid group and group presentation generated by elements corresponding to horizontal trisecants. This approach does not apply to links nor tangles because it requires that when counting trisecants, we have the same number of points at each level. For general tangles, trisecants passing through one component twice may occur. Free links can be obtained from tangles by attaching two end points of each component. We shall construct an invariant of free links and free tangles valued in groups as follows: we associate elements in the groups with 4-valent vertices of free tangles(or free links). For a free link with enumerated component, we `read' all the intersections when traversing a given component and write them as a group element. The problem of `pure crossings' of a component with itself by using the following statement: {\em if two diagrams with no pure crossings are equivalent then they are equivalent by a sequence of moves where no intermediate diagram has a pure crossing.} This statement is a result of a sort that an equivalence relation within a subset coincides with the equivalence relation induced from a larger set and it is interesting by itself\\
{\bf Key words and phrases :} free link diagram, group presentation, free link invariant, bracket invariant\\
{\bf 2010 Mathematics Subject Classification: 57M25, 57M27 }
\end{abstract}

\section{Introduction and basic defnitions}
Virtual knots were introduced by Kauffman \cite{Ka} as knots in thickened surfaces considered up to isotopy and stabilization/destabilization. Virtual knots have simple diagrammatic descriptions by virtual diagrams and Reidemeister moves and detour moves. If we forget every structure at classical crossings except for framing, we get a simplification of virtual knots, called {\it free knots}. But it can be shown that there are non-trivial free knots by parity \cite{Ma3} and free knots are non-trivial enough.

Since the discovery of parity by the second named author \cite{Ma1}, the following principle was established:

{\em if a diagram $K$ is complicated enough then it realizes itself, i.e., it appears as a subdiagram in any
diagram $K'$ equivalent to $K$. }

This approach is realized by using invariants of free links valued in diagrams of free links.
One of them, called the parity bracket \cite{Ma1}, possesses the property $[K]=K$ for a diagram which is complicated enough is equal to this diagram itself. In particular, if $K'$ is equivalent to $K$ then $[K']=[K]=K$, which means by construction, that $K$ appears inside $K'$.
 This effects very similarly to the case of free groups: if, say, the word $abcba$  in $\mathbb{Z}*\mathbb{Z}*\mathbb{Z}$ is irreducible then it appears in any word equivalent to it (say, $abaa^{-1}b^{-1}bcbc^{3}c^{-3}a)$.
 One way of constructing a bridge between free knots and free groups possessing similar nice properties was undertaken in \cite{Ma2},\cite{MaNi}: starting from an element of the classical braid group, we got an element of some free group closely related to it via some group $G_{n}^{3}$. This works for the case of classical braids. However, the invariant constructed in \cite{MaNi} is not arranged for the case of tangles.
 
The aim of the present paper is to construct invariants of free $n-n$ tangles and links valued in free groups.
 
The main difference between arbitrary $n-n$ tangles and braids is the existence of pure crossings (between
a component and itself). However, this is not the only difference: there are diagrams of tangles with no pure crossings which cannot be represented by a closure of a braid, see Fig.~\ref{exa-tangle1}. In this Figure, intersections of two edges are classical crossings and intersections with circle are virtual crossings. There is an arc such that a part of the arc bounds a disk on the plane.

 
 \begin{figure}[h!]
\begin{center}
 \includegraphics[width = 6cm]{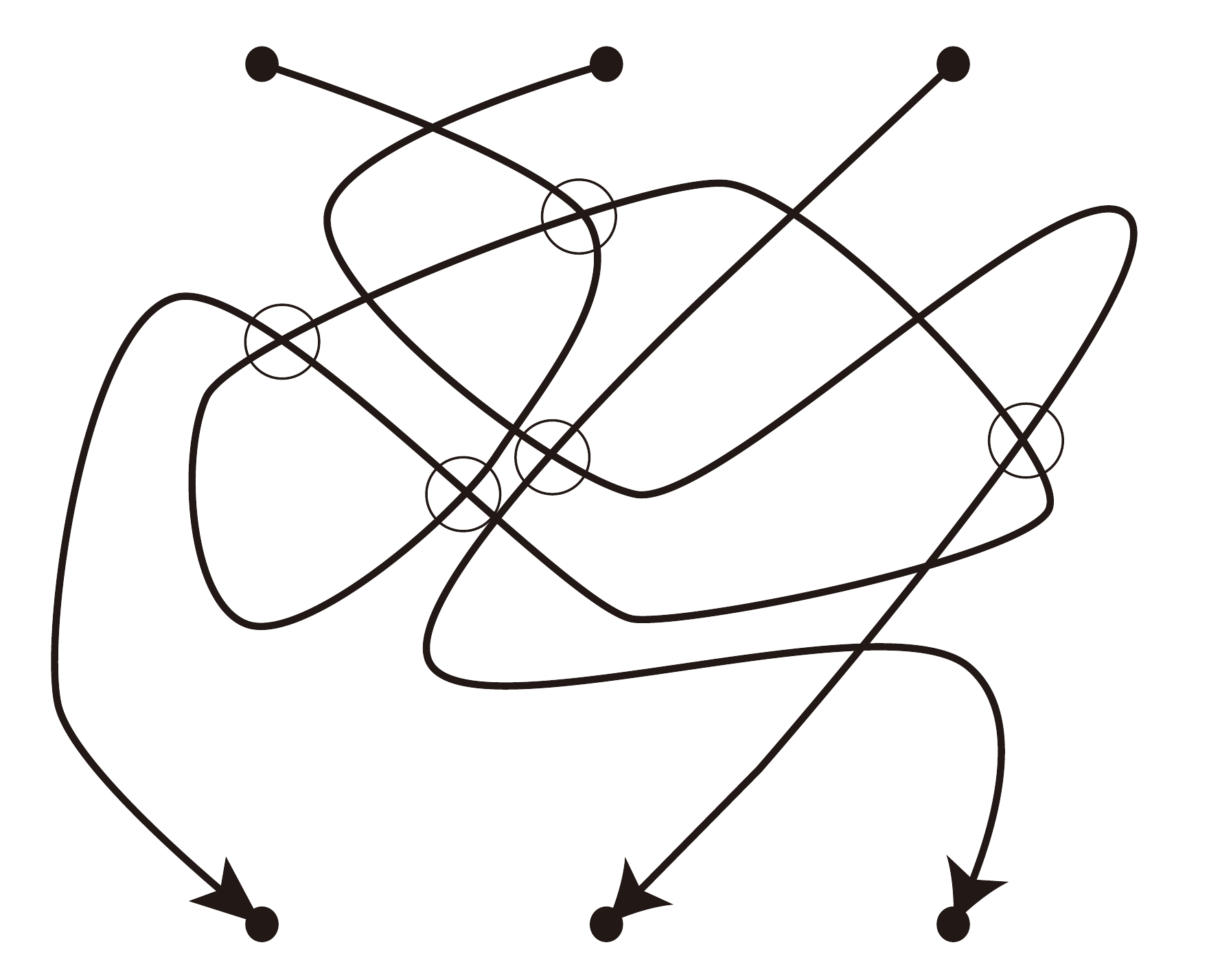}

\end{center}

 \caption{Free tangle diagram}\label{exa-tangle1}
\end{figure}
 
Unlike braids, $n-n$ tangles do not possess a group structure: there are no inverse elements, thus, they
are more difficult to work with.
 
We overcome this difficulty by showing that
 
{\em if two diagrams with no pure crossings are equivalent then they are equivalent by a sequence
of moves where no intermediate diagram has a pure crossing.}
 
 Theorem~\ref{Proj-link} is of an interest of its own: it is a result of the sort that the equivalence relation on a subset induced from a larger set coincides with the equivalence relation on the same subset induced by local equivalence relations within the subset.

Among the results of this sort, we mention the following ones:

1) if two classical links are equivalent as virtual links, then they are equivalent as classical links.

2) if two classical braids are equivalent as tangles, then they are equivalent as braids.
 
3) if two classical braids are equivalent as virtual braids, then they are equivalent as classical
braids.
 
The first result \cite{GoPoVi} and the second result are proved by classical methods (fundamental group). The third one is firstly proved in \cite{FeRiRo} and in \cite{El} it is proved by parity method .
 
On the other hand, the statement about virtual braids and virtual tangles still remains a conjecture.

Notice that a tangle may contain an arc such that a part of the arc bounds a circle on the plane, see Fig.~\ref{exa-tangle1-1}. 
 \begin{figure}[h!]
\begin{center}
 \includegraphics[width = 6cm]{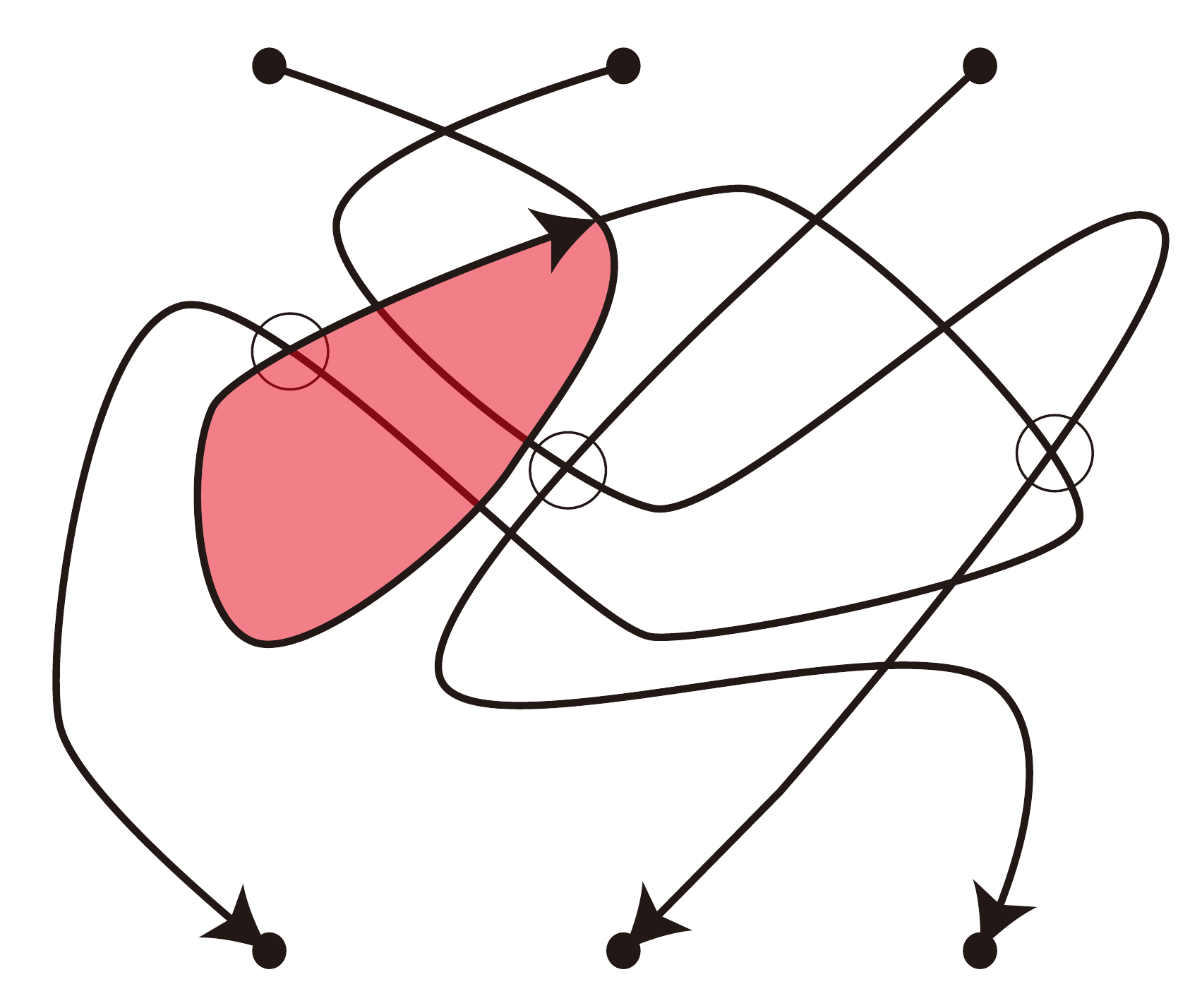}

\end{center}

 \caption{Free $n-n$ tangle diagram with a reverse arc}\label{exa-tangle1-1}
\end{figure}
Because of that, roughly speaking, there are no ``universal" orders for crossings of tangles and it is difficult to represent tangles by elements in some group as braids.

In the present paper for a $n-n$ tangle we fix a component of the tangle and read a `word' valued in group according to the orientation of the arc such that each character of the word is associated to a classical crossing. In Section 3 we are going to show that this word is well-defined, i.e, the map from the set of $n-n$ tangles to the group is an invariant under Reidemeister moves. For a link $L$, a $n-n$ tangle $T_{L}$ can be obtained by cutting each component of the link at fixed points $\{p_{i}\}$ and a word can be obtained from the $n-n$ tangle $T_{L}$. The word is an invariant under Reidemeister moves with corrections provided by the cut locus $\{p_{i}\}$.

Now we introduce basic definitions: By a {\it framed 4-graph with endpoints}  we mean a graph satisfying the followings:
 \begin{enumerate}
\item{every vertex is a 4-valent vertex except for $2n$ for some $n \in \mathbb{N} \cup \{0\}$ vertices which are 1-valent vertices.} 
\item{for each 4-valent vertex we fix a way of splitting of the four emanating half-edges into two pairs of edges called (formally) {\it opposite}}.
\end{enumerate}
Intersection of different edges of the framed 4-graph with endpoints in interior points are called {\it virtual crossings} and it is denoted by an intersection inside a circle, see Fig.~\ref{vir-cro}.

\begin{figure}[h!]
\begin{center}
 \includegraphics[width = 3cm]{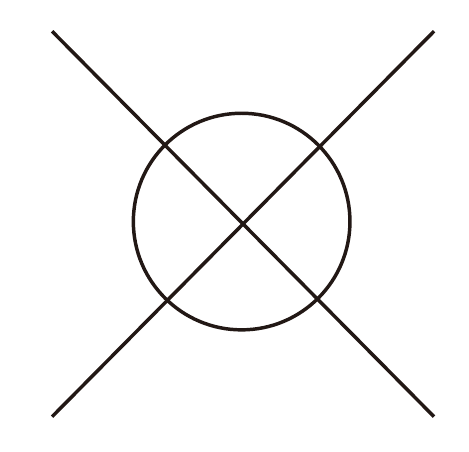}

\end{center}

 \caption{Virtual crossing}\label{vir-cro}
\end{figure}
One of examples of framed 4-graphs with endpoints is in Fig.~\ref{exa-sheeted-graph}. By abusing notation, we say ``graph" not only for graphs but also for disjoint collection of circles and for split sums of graphs with collections of circles. By vertices of graphs we also mean genuine vertices of graph components. We also admit empty graph as a framed 4-graph with endpoints.
 

We call 1-valent vertex in $\mathbb{R} \times \{1\}$(or in $\mathbb{R} \times \{0\}$) a {\it upper(or lower) point}. 
\begin{figure}[h!]
\begin{center}
 \includegraphics[width = 6cm]{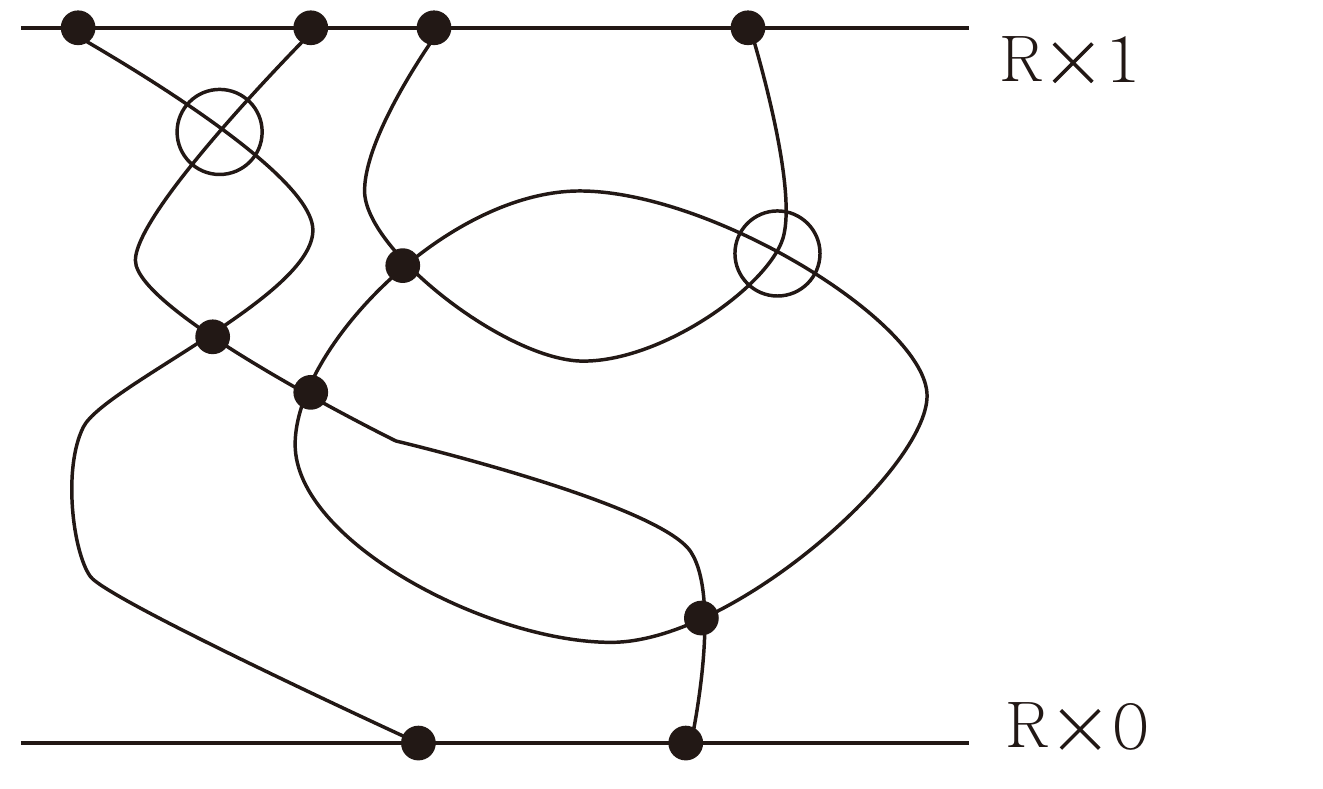}

\end{center}

 \caption{Framed 4-graph with endpoints}\label{exa-sheeted-graph}
\end{figure}

\begin{dfn}
By a {\it unicursal component} of a framed 4-graph with endpoints we mean an equivalence class on the set of edges of the graph: Two edges $e$, $e'$ are equivalent if there exists a collection of edges $e=e_{1}, \cdots, e_{k} = e'$ and a collection of 4-valent vertices $v_{1}, \cdots, v_{k-1}$ (some of them may coincide) of the graph such that edges $e_{i}$, $e_{i+1}$ are opposite to each other at the vertex $v_{i}$. By a {\it unicursal circle} of a framed 4-graph with endpoints we mean a unicursal component such that $v_{1} = v_{k-1}$.
\end{dfn}

A {\it virtual tangle diagram}  is a generic immersion of a framed 4-graph with endpoints in $\mathbb{R} \times I$ with each 4-valent vertex endowed with a classical crossing structure and every 1-valent vertex is upper or lower point. A {\it virtual (link) diagram} is a virtual tangle diagram without endpoints.

 A {\it virtual tangle} is an equivalence class of virtual tangle diagrams by usual Reidemeister moves in Fig.~\ref{cla-Rmoves} and the {\it detour move}. 
  \begin{figure}[h!]
\begin{center}
 \includegraphics[width = 8cm]{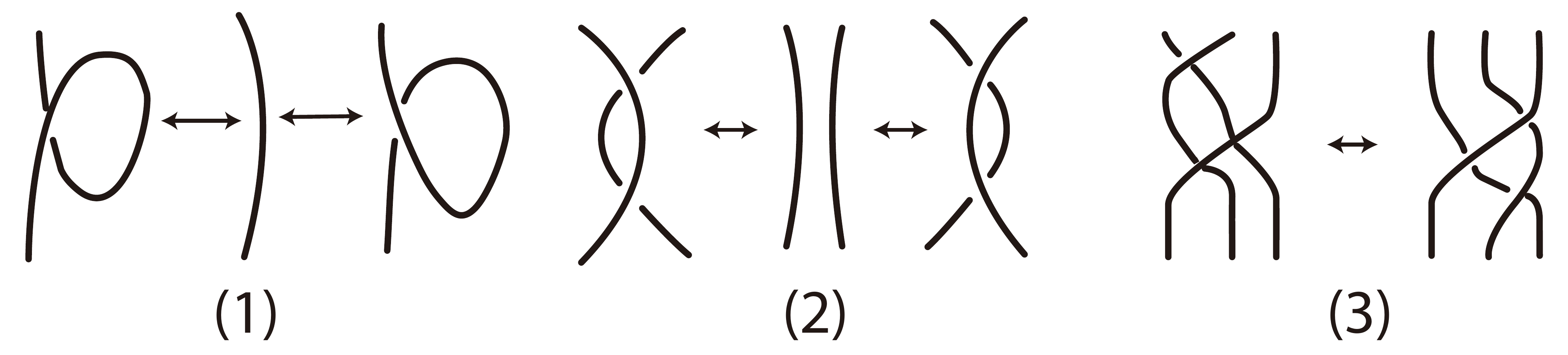}

\end{center}

 \caption{Reidemeister moves}\label{cla-Rmoves}
\end{figure}
 
 The detour move changes the immersion of an edge of the graph: one takes the edge fragment drawn on the plane which has only virtual crossings and redraws it arbitrarily in a generic way with all new crossings specified as virtual, see Fig.~\ref{exa-detour}. 
  \begin{figure}[h!]
\begin{center}
 \includegraphics[width = 8cm]{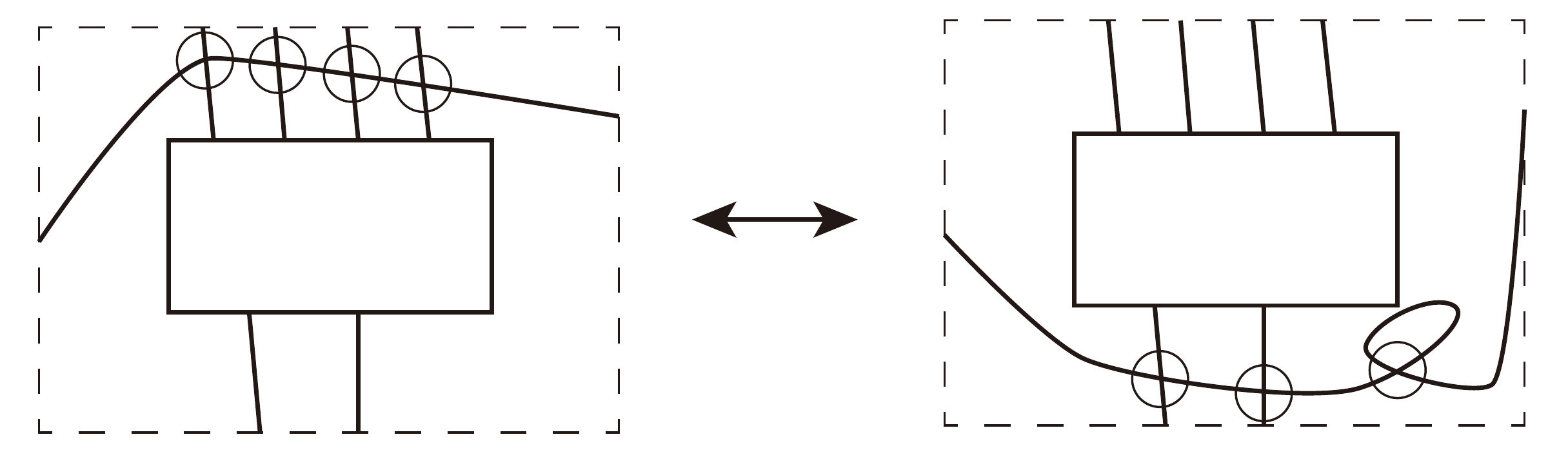}

\end{center}

 \caption{Detour move}\label{exa-detour}
\end{figure}

 A {\it virtual link} is a virtual tangle without endpoints.  A {\it virtual knot} is a virtual link with one unicursal component. 

 
By an {\it $n-n$ virtual tangle}, we mean a virtual tangle with lower points $\{p_{1}^{0}, \cdots, p_{n}^{0}\}$ and upper points $\{p_{1}^{1}, \cdots, p_{n}^{1}\}$ such that each component has end points $p_{i}^{0}$ and $p_{i}^{1}$. Note that an $n-n$ virtual tangle has no unicursal circles.

  If the components of a virtual tangle are numbered and the numbers of components preserve under Reidemeister moves and detour move, then the virtual tangle is {\it enumerated.} In a similar way, we can define enumerated tangle diagrams.
  
 Now let us forget ``over/under" information from virtual links, remembering frame for each classical crossings and roughly speaking, this is a free link. Free links are defined as follows. 

\begin{dfn}
 A {\it free tangle} is an equivalence class of framed 4-graphs with endpoints modulo Reidemeister moves for free diagrams in Fig.~\ref{Rmoves}. A {\it free link} is a free tangle without endpoints. A {\it free knot} is a free link with one unicursal circle. 
 
 By an {\it $n-n$ free tangle}, we mean a free tangle with lower points $\{p_{1}^{0}, \cdots, p_{n}^{0}\}$ and upper points $\{p_{1}^{1}, \cdots, p_{n}^{1}\}$ such that each component has end points $p_{i}^{0}$ and $p_{i}^{1}$. 
 
 If the components of a free tangle are numbered and the numbers of components preserve under Reidemeister moves for free diagrams, then the free tangle is {\it enumerated.} In a similar way, we can define enumerated free tangle diagrams.
 \begin{figure}[h!]
\begin{center}
 \includegraphics[width = 8cm]{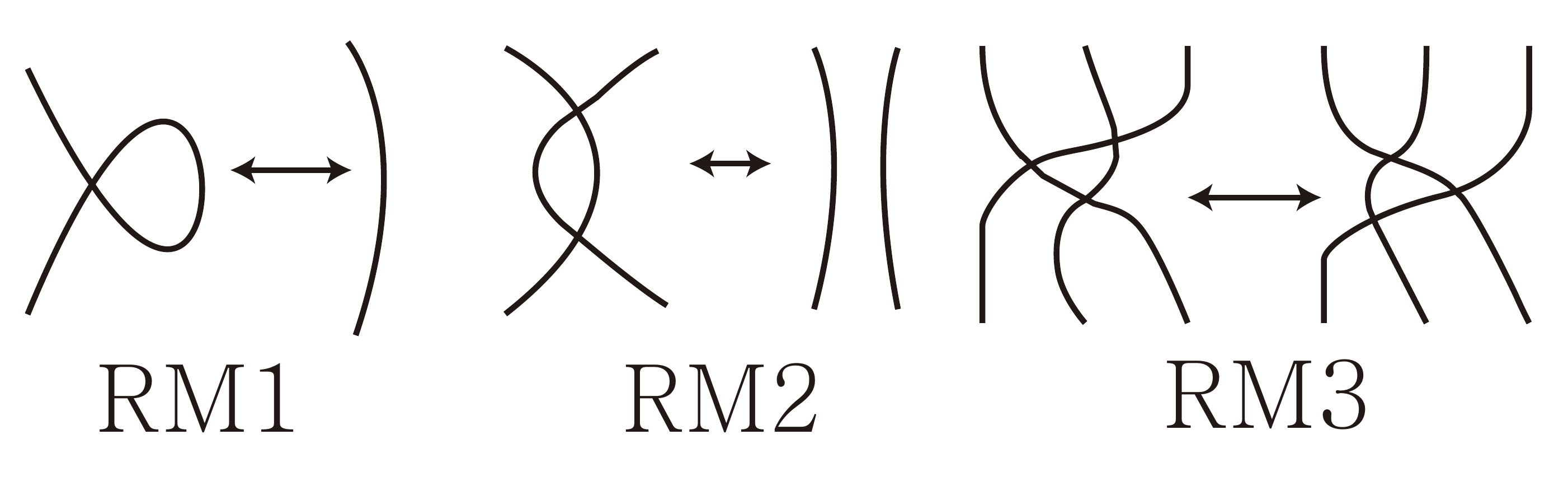}

\end{center}

 \caption{Reidemeister moves for free diagrams}\label{Rmoves}
\end{figure}

\end{dfn}

Free tangles are closely related to flat virtual tangles.
 
 \begin{dfn}
 A {\it flat virtual tangle} is an equivalence class of virtual tangles modulo changing over/under crossing structure. A {\it flat virtual link} is a flat virtual tangle without endpoints. 
 
 By an {\it $n-n$ flat virtual tangle}, we mean a flat virtual tangle with lower points $\{p_{1}^{0}, \cdots, p_{n}^{0}\}$ and upper points $\{p_{1}^{1}, \cdots, p_{n}^{1}\}$ such that each component has end points $p_{i}^{0}$ and $p_{i}^{1}$.
 
 If the components of a flat virtual tangle are numbered and the numbers of components preserve under Reidemeister moves and detour move, then the flat virtual tangle is {\it enumerated.} In a similar way, we can define enumerated flat virtual tangle diagrams. 
 
 By {\it virtualization} we mean a move for flat virtual tangles in Fig.~\ref{Virtualization}.
 
  \begin{figure}[h!]
\begin{center}
 \includegraphics[width = 5cm]{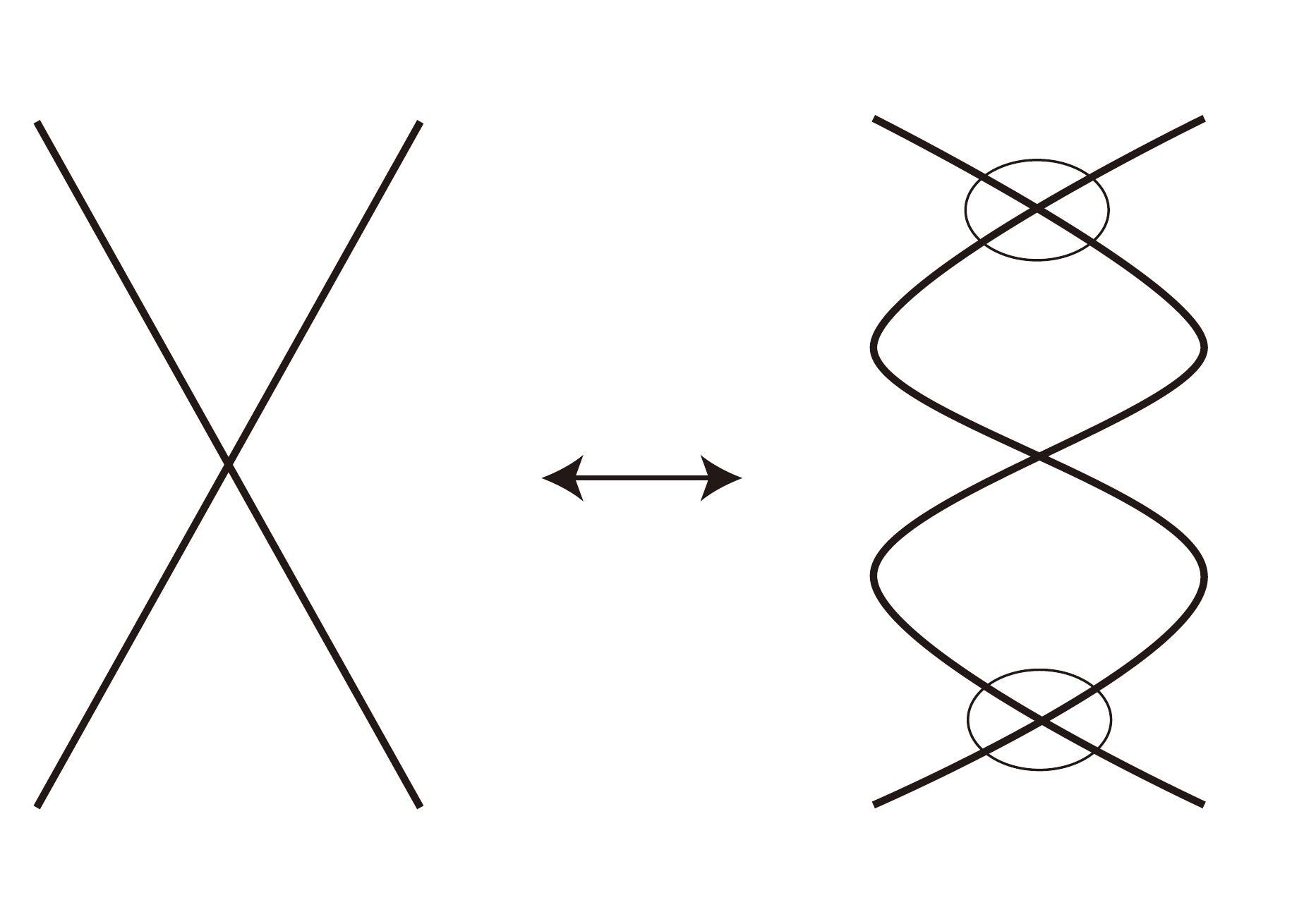}

\end{center}

 \caption{Virtualization}\label{Virtualization}
\end{figure}

\end{dfn}

\begin{rem}
Every notion($n-n$ tangle, link, knot, braid, etc.) is just a partial case of tangles. On the other hand, free tangles($n-n$ tangles, links, knots, etc.) are equivalence classes of flat tangles modulo virualization. 
\end{rem}

\begin{rem}
Equivalence relation for virtual(flat, free) braids should be defined with an extra care because there are in fact two equivalence relations: the one within the category of tangles and the one using braid-like moves.
\end{rem}

The paper is organized as follows: in Section 2, we define a bracket invariant for free tangles and prove Theorem~\ref{Proj-link}. In Section 3, we are going to introduce an invariant for enumerated oriented $n-n$ free tangles and free links valued in free product of $\mathbb{Z}_{2}$ and show an example for which the value of the invariant is not trivial.

\section{Free link diagrams without pure crossings}


\begin{dfn}

Let $L = L_{1} \cup \cdots \cup L_{n}$ be an enumerated free link(or tangle) diagram. A crossing $c$ is called a crossing of {\it type $(i,j)$ } if two arcs of $c$ are part of $L_{i}$ and $L_{j}$ respectively. A classical crossing $c$ of type $(i,i)$ for some $i \in \{1,2, \cdots n\}$ is called a {\it pure crossing}.

\end{dfn}


Let us define an equivalence relation $\cong$ on enumerated free tangle diagrams as follows: two enumerated diagrams $T$ and $T'$ are {\em g-equivalent} if there is a sequence of enumerated free tangle diagrams obtained by applying one of Reidemeister moves for free tangles between $T$ and $T'$ such that each diagram has no pure crossings. Note that the first Reidemeister move for free diagrams cannot be applied. Let $\mathbb{T}$ be the set of equivalence classes of enumerated free tangle diagrams with no pure crossings modulo $\cong$. Consider the linear space $\mathbb{Z}_{2}[\mathbb{T}]$.

We define a bracket for enumerated free links valued in $Z_{2}[\mathbb{T}]$. Let $T$ be an enumerated free tangle diagram with $n$ components. 
Let us consider pure crossings of $T$. 
For each pure crossing, there are two splicings, see Fig.~\ref{splicing}. Herewith, the rest of the diagram remains unchanged. We may then consider further splicings of $T$ at several crossings. The result of splicings $T_{s}$ would have at least $n$ components.
 \begin{figure}[h!]
\begin{center}
 \includegraphics[width = 6cm]{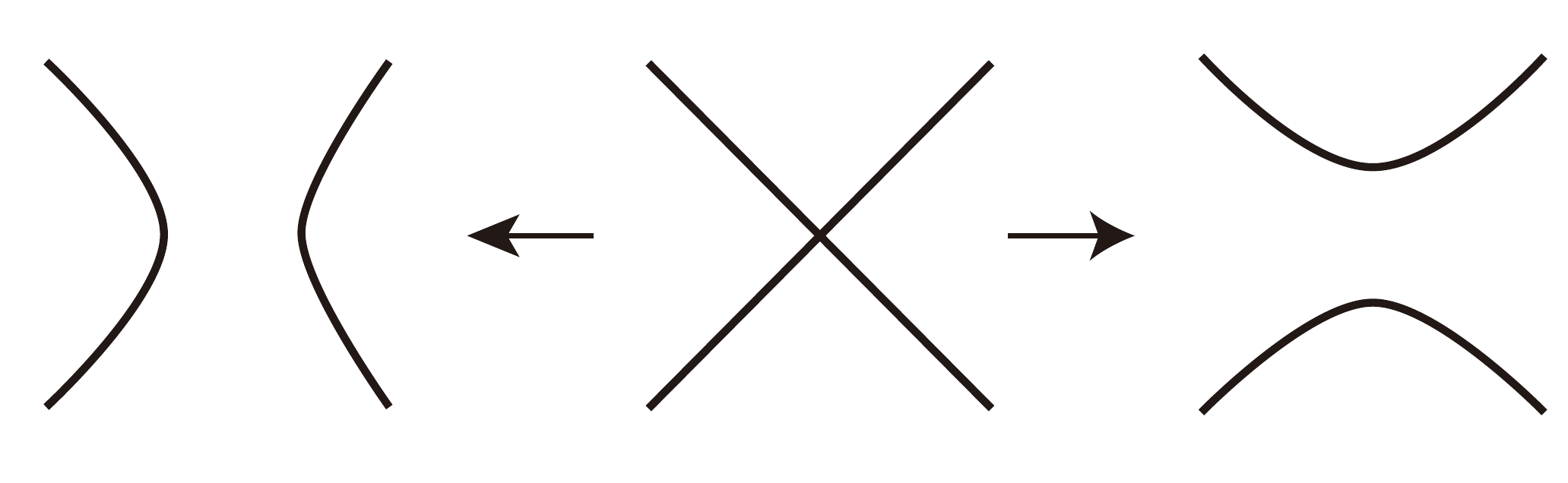}

\end{center}

 \caption{Splicings of a crossing}\label{splicing}
\end{figure}

\begin{rem}
Let $T$ be an enumerated free tangle diagram. If $T_{s}$ has exactly $n$ components, then it is possible to enumerate components which agrees with enumeration of $T$ because we splice only pure crossings and any two different component cannot be connected by the operation. Then if $T_{s}$ has exactly $n$ components, then $T_{s}$ is enumerated such that it agrees with enumeration of $T$. 
\end{rem}

Let us consider the following sum
\begin{center}
\begin{equation}
[T] = \Sigma_{\{s-pure,n~comp.\}} T_{s} \in \mathbb{Z}_{2}[\mathbb{T}]
\end{equation}
\end{center}

 which is taken over all splicings in all pure crossings, and only those summands are taken into account where $T_{s}$ has exactly $n$ componens. Thus if $T$ has $m$ pure crossings, then $[T]$ will contain at most $2^{m}$ summands, and if $T$ has no pure crossings, then we shall have exactly one summand. This bracket is similar to the parity bracket which is introduced in~\cite{Ma1}.
 
 \begin{lem}\label{lem-freelink}
The bracket $[ ~\cdot~]$ for enumerated free tangles is an invariant under Reidemeister moves for free diagrams.
\end{lem}

\begin{proof}
Let $T$ and $T'$ be two enumerated free tangle diagrams such that $T'$ is obtained from $T$ by applying one of Reidemeister moves for free diagrams. If $T'$ is obtained by applying the first Reidemeister move for free diagrams, for a crossing which the first Reidemeister move can be applied to, there is the only one way to splice the crossing, otherwise the number of components is more than or equal to $n+1$, see Fig.~\ref{splice-R1}. Hence $[T] = [T']$.
 \begin{figure}[h!]
\begin{center}
 \includegraphics[width = 6cm]{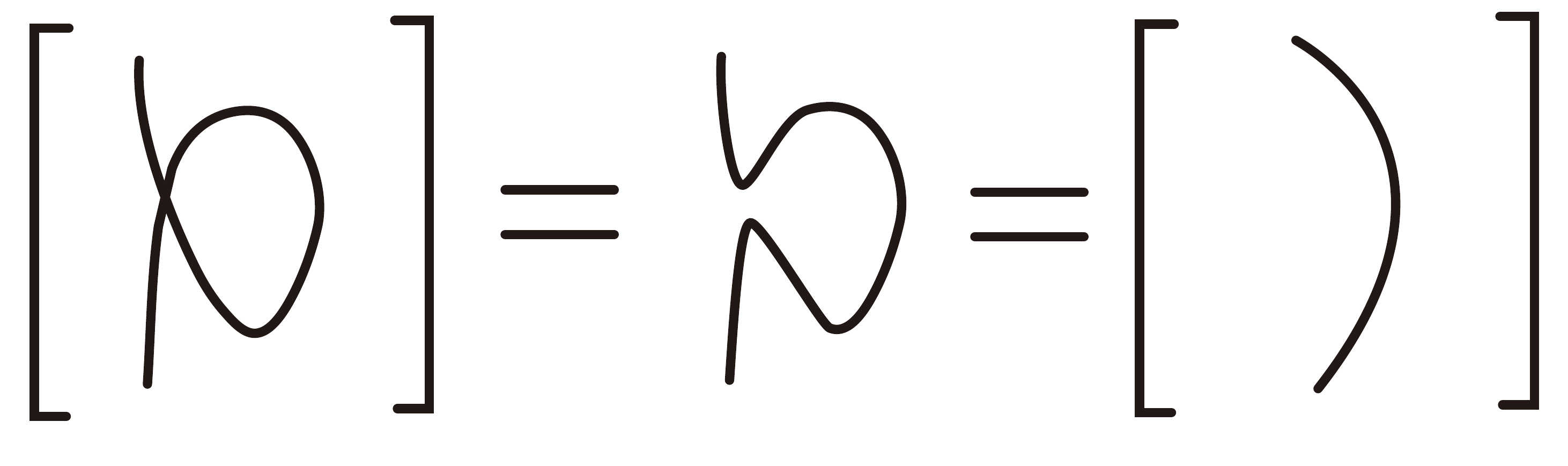}

\end{center}

 \caption{Splicing for RM1}\label{splice-R1}
\end{figure}
 Suppose that $T'$ is obtained from $T$ by applying the second Reidemeister move for free diagrams. 
 If both crossings are mixed, then we don't splice them and  $[T] = [T']$ because $T$ and $T'$ are equivalent in $\mathbb{T}$ and the bracket is defined on $\mathbb{Z}_{2}[\mathbb{T}]$, see Fig.~\ref{splice-R2-1}. 
 
   \begin{figure}[h!]
\begin{center}
 \includegraphics[width = 8cm]{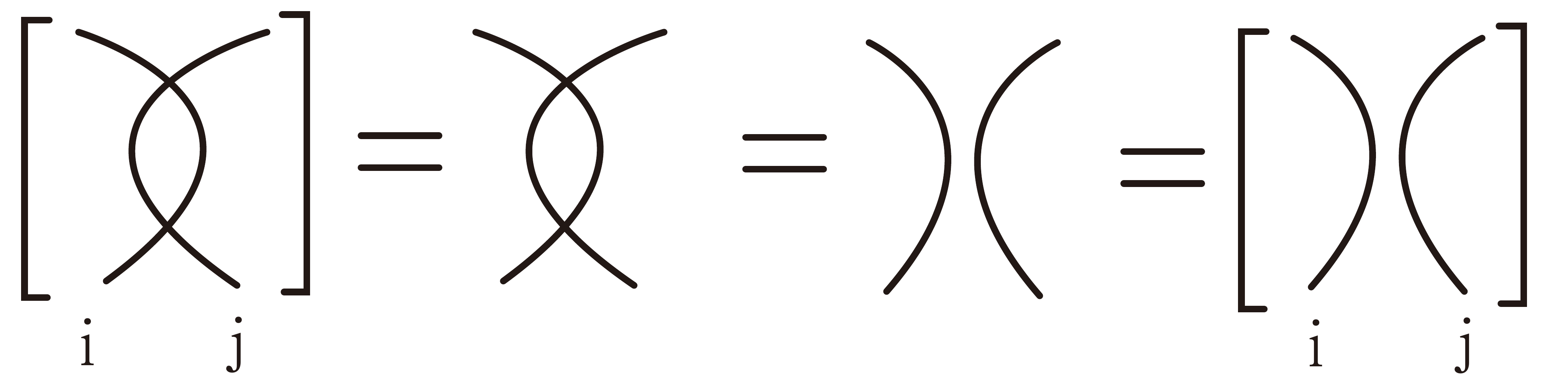}

\end{center}

 \caption{Splicing for RM2 between different components}\label{splice-R2-1}
\end{figure}
If we apply a  second Reidemeister move which involves two pure crossings (see Fig.~\ref{splice-R2}), then two summands (the second one and the third one) cancel each other.
Since we only consider $T_{s}$ with $n$ components, the last diagram of the first line in Fig.~\ref{splice-R2} cannot appear. Therefore  $[T] = [T']$.

 \begin{figure}[h!]
\begin{center}
 \includegraphics[width = 8cm]{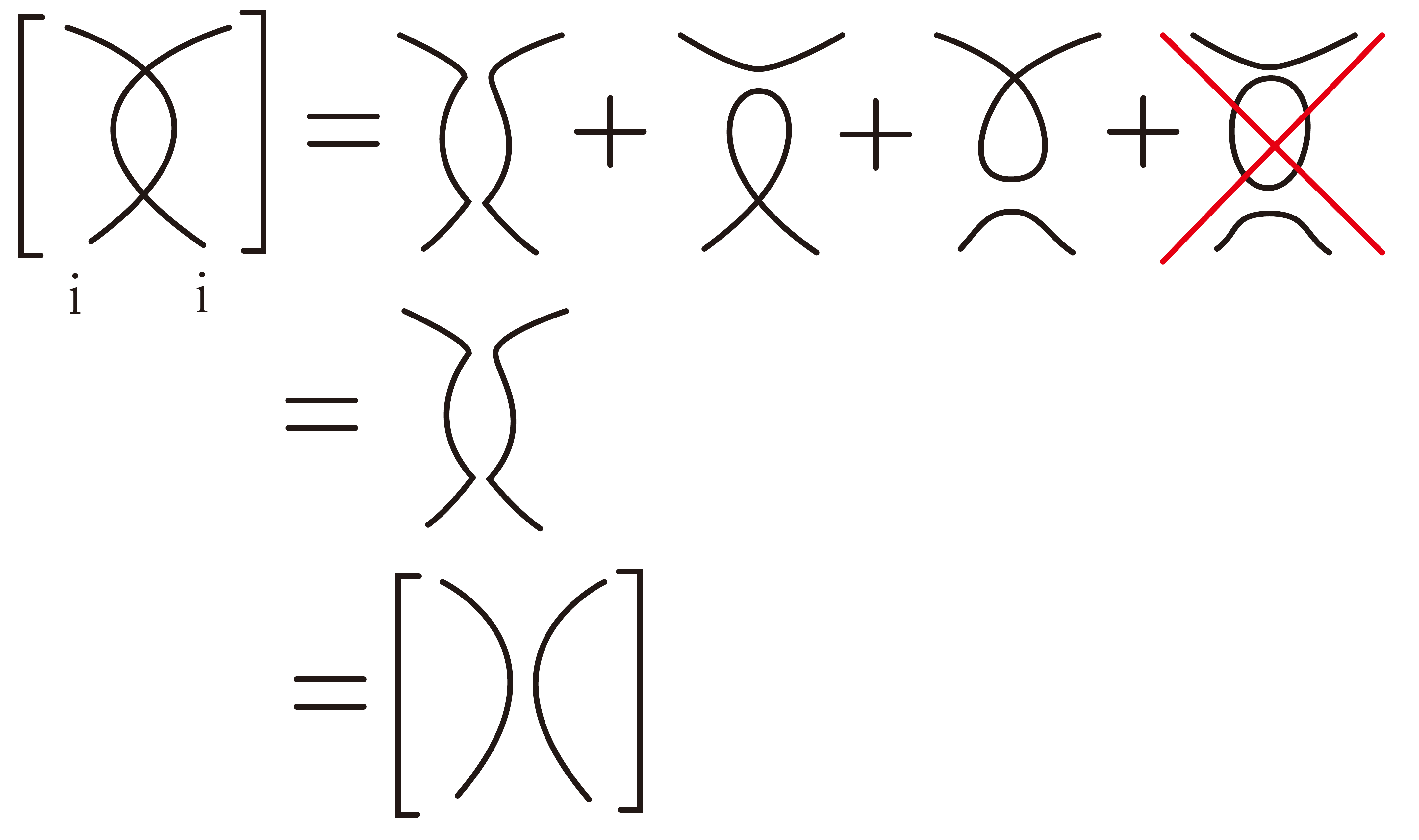}

\end{center}

 \caption{Splicing for RM2 between a component and itself.}\label{splice-R2}
\end{figure}
 Finally, consider the case of the third Reidemeister move for free diagrams, say the move is related to $i,j,k$ components. If $i \neq j \neq k$, then crossings which are contained in the move are not spliced and  $[T] = [T']$ because $T$ and $T'$ are equivalent in $\mathbb{T}$ and the bracket is defined on $Z_{2}[\mathbb{T}]$. In the case of $i=j=k$, we can show that $[T] = [T']$ by Fig.~\ref{splice-R3-1}.(Diagrams in the same color are equivalent in $\mathbb{T}$.) Note that the first diagrams of splicing of $T$ and $T'$ do not happen because they have more than $n$ components.
  \begin{figure}[h!]
\begin{center}
 \includegraphics[width = 8cm]{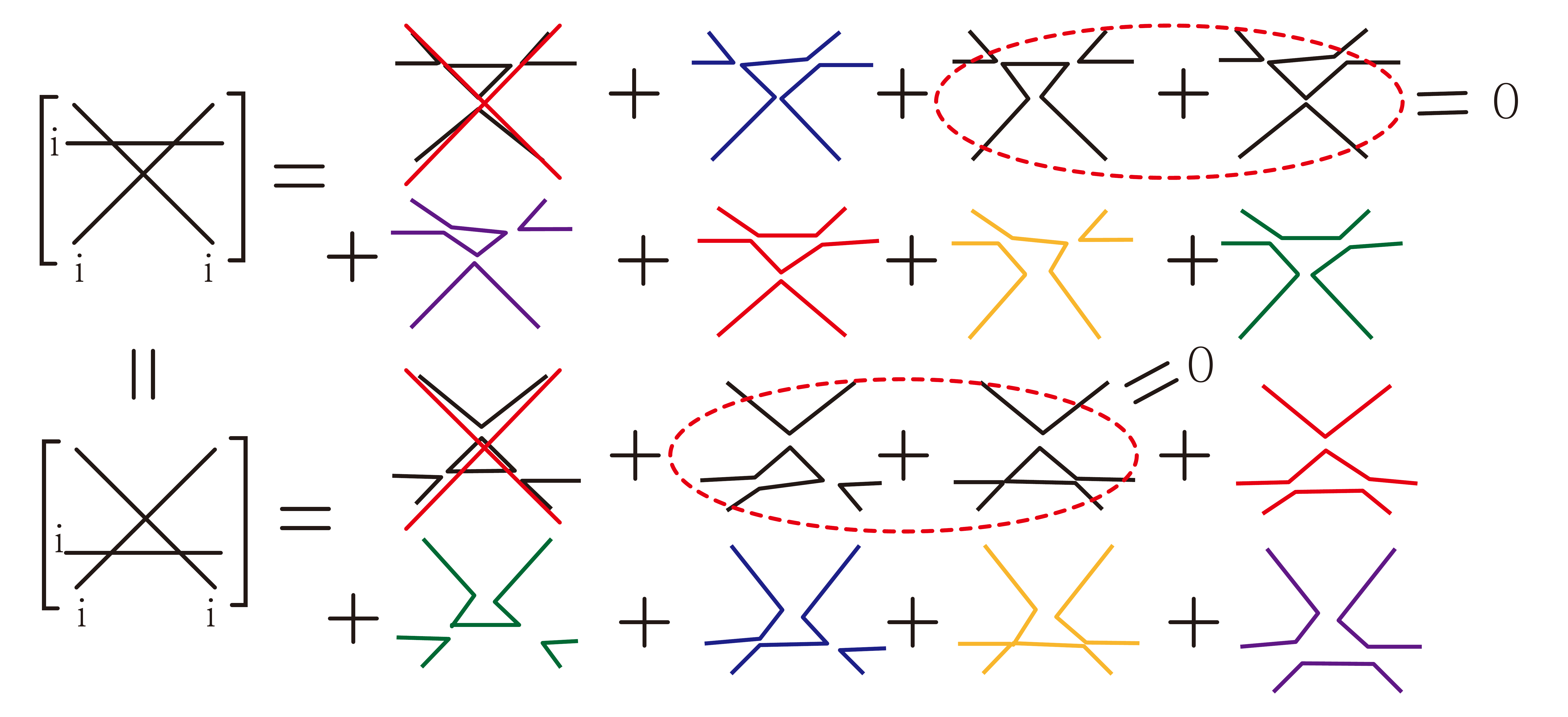}

\end{center}

 \caption{Splicing for the third Reidemeister move, $i=j=k$}\label{splice-R3-1}
\end{figure}
 Now we consider that the case of $i\neq j=k$, see Fig.~\ref{splice-R3}. Since the bracket is valued in $Z_{2}[\mathbb{T}]$, the second diagrams of the first and the second lines in Fig.~\ref{splice-R3} and hence $[T] = [T']$. 
   \begin{figure}[h!]
\begin{center}
 \includegraphics[width = 8cm]{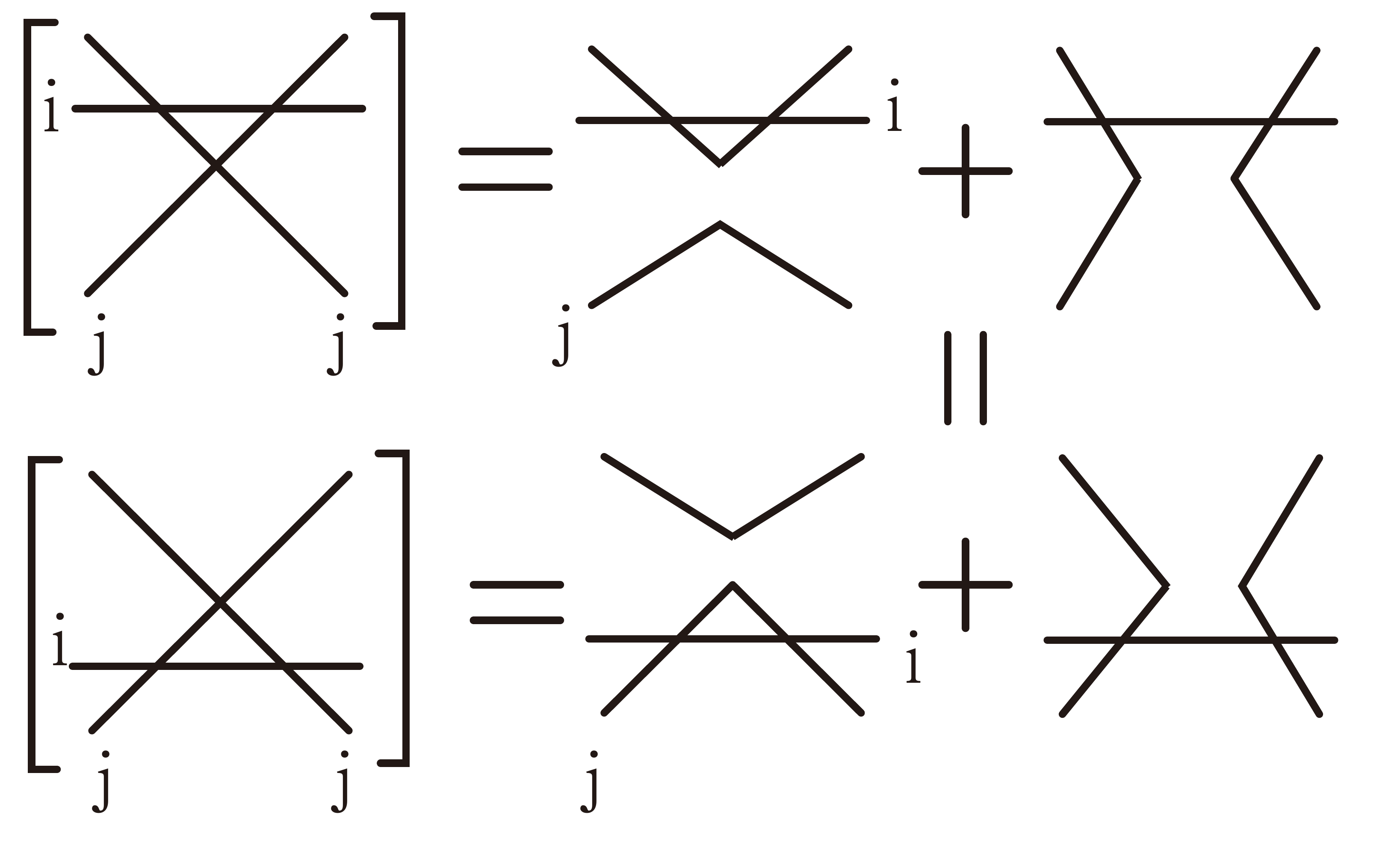}

\end{center}

 \caption{Splicing for the third Reidemeister move, $i\neq j=k$}\label{splice-R3}
\end{figure}

\end{proof}

Since free links are a part of free tangles, we can easily get the following Corollary:
 
 \begin{cor}\label{cor-freetangle}
The bracket $[ ~\cdot~]$ for enumerated free links is an invariant under Reidemeister moves for free diagrams.
\end{cor}

We are now going to prove the following Theorem:
\begin{thm}\label{Proj-link}
If two enumerated free tangle diagrams with no pure crossings are equivalent then they are g-equivalent.\end{thm}

\begin{proof}
For two enumerated free tangle diagrams $T$ and $T'$ with no pure crossings if $T'$ can be obtained from $T$ by Reidemeister moves for free diagrams, then $[T] = [T']$. Since $T$ and $T'$ have no pure crossings, $T = [T] = [T'] = T'$. By definition of $[~\cdot~]$, $T \cong T'$ and the theorem is proved. 

\end{proof}

\section{Free link invariant valued in group}
In the present section, we are going to construct the main invariant. For free $n-n$ tangles it is valued in the free group, hence, is very easy to calculate. When passing to free links, we get some equivalence classes which actually resume to conjugations and permutations of indices. 
\begin{dfn}
 Let $L = L_{1} \cup L_{2} \cup \cdots \cup L_{n}$ be an enumerated free link(or tangle) diagram. If the number of crossings of type $(i,j)$ is even for all $i,j$ such that $i \neq j$, then we call $L$ a diagram \textit{in a good condition.}
\end{dfn}

Let $T = T_{1} \cup T_{2} \cup \cdots \cup T_{n}$ be an enumerated $n-n$ free tangle diagram in a good condition without pure crossings. Let us orient the tangle $T$ from $\mathbb{R} \times \{0\}$ to $\mathbb{R} \times \{1\}$. Denote boundary point of $T_{i}$ in  $\mathbb{R} \times \{0\}$ and $\mathbb{R} \times \{1\}$ by $p_{i}^{s}$ and $p_{i}^{e}$.

 For each classical crossing $c$ of $T$ of type $(i,j)$ and for $ k \in \{1,2,\cdots, n\} \backslash \{i,j\}$, define $lk^{i}_{c}(k)$ by the sum of number of all crossings of type $(i,k)$ on the $T_{i}$ from the points $p_{i}^{s}$ to the crossing $c$. Define $lk_{c}(k) = lk^{i}_{c}(k) +lk^{j}_{c}(k)$ modulo $\mathbb{Z}_{2}$. Note that $lk_{c}$ can be considered as a map from $\{1,2,\cdots, n \} \backslash \{i,j\} $ to $\mathbb{Z}_{2}$.
 
 Fix $i\neq j, i,j\in \{1,\dots, n\}$. Let $\{c_{1}, \cdots, c_{m}\}$ be the set of classical crossings of type $(i,j)$ such that for each $k , l \in \{1, 2, \cdots, m\}$, $k < l$ if and only if we meet $c_{k}$ earlier than $c_{l}$ when we follow $i$-th component from the point $p_{i}^{s}$. For example, in Fig.~\ref{exa-enu-cro}, there are 4 crossings of type $(1,2)$. With respect to the first component they are enumerated by $\{c_{1},c_{2},c_{3},c_{4}\}$. With respect to the second component they are enumerated by $\{d_{1},d_{2},d_{3},d_{4}\}$.

  \begin{figure}[h!]
\begin{center}
 \includegraphics[width = 6cm]{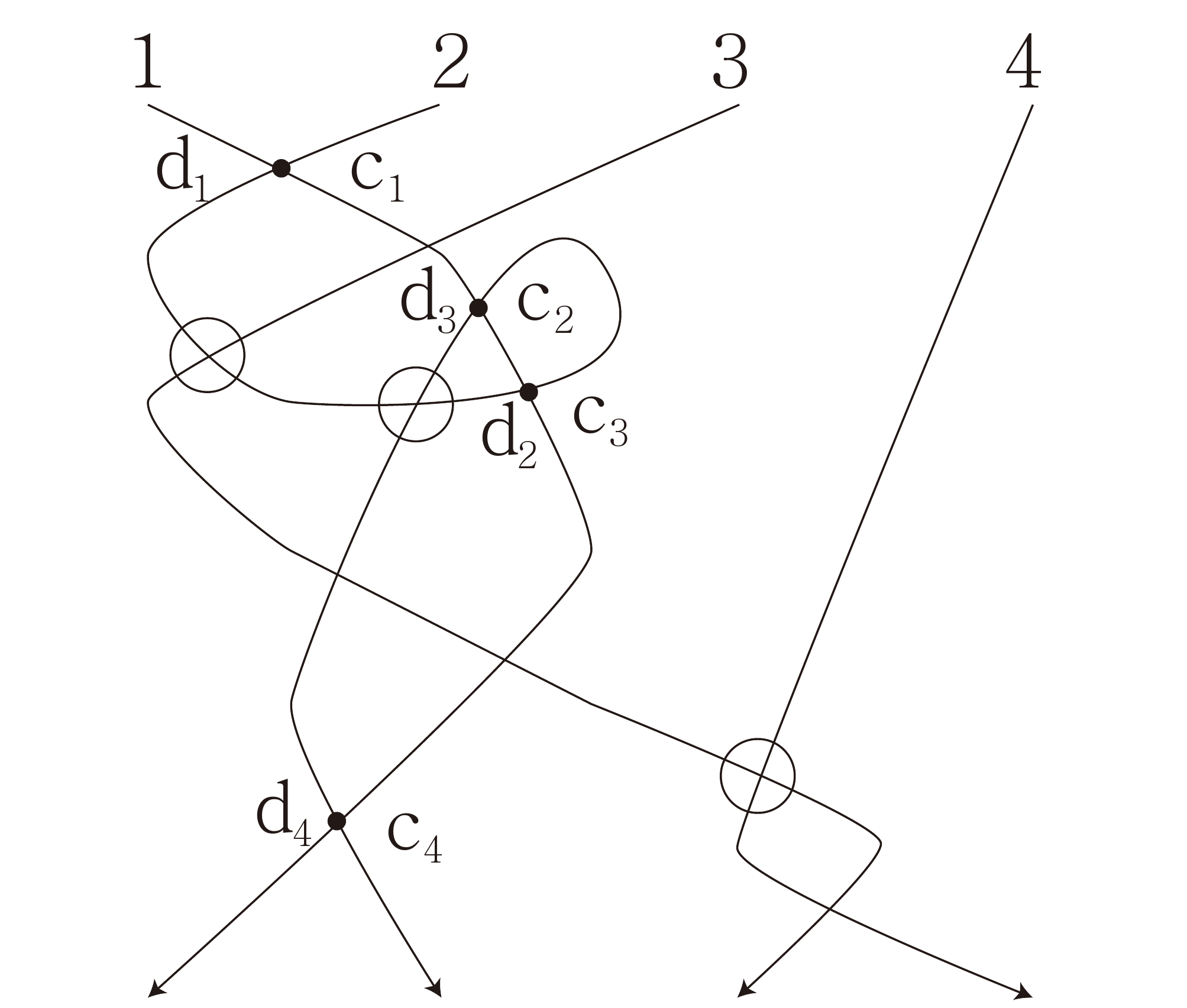}

\end{center}

 \caption{}\label{exa-enu-cro}
\end{figure}

 Define a group presentation $G_{n}^{(i,j)}$ generated by $\{ \sigma ~|~ \sigma : \{1,2,\cdots n\} \backslash \{i,j\} \rightarrow \mathbb{Z}_{2} \}$ with relations $\{ \sigma^{2} = 1\}$. Let us define an action of $aut(G_{n}^{(i,j)})$ on $G_{n}^{(i,j)}$ by $f \cdot g = f(g)$ for $f \in aut(G_{n}^{(i,j)})$ and for $g \in G_{n}^{(i,j)}$. For $l \in \{1,\cdots,n\} \backslash \{i,j\}$, let $f_{l}$ be an automorphism on $G_{n}^{(i,j)}$ such that 

\begin{center}
$f_{l}(\sigma)(x) = \left\{
\begin{array}{cc} 
     \sigma(x) & \text{if}~x \neq l \\
       \sigma(x)+1
 &  \text{if}~x=l. \\
   \end{array}\right.$

\end{center}
Define two word $w$ and $w'$ in $G_{n}^{(i,j)}$ are {\it slide-equivalent} in $G_{n}^{(i,j)}$ if $w' = f \cdot w$ for some $ f $ in the subgroup $< \{ f_{l} \}_{l \in \{1,2, \cdots m\}} > $ of $ aut(G_{n}^{(i,j)})$.

 Note that $G_{n}^{(i,j)} \cong_{h} \mathbb{Z}_{2}^{*2^{n-2}}$ where $h$ is an isomorphism from $G_{n}^{(i,j)}$ to $ \mathbb{Z}_{2}^{*2^{n-2}}$. By definition, each $lk_{c}$ is one of the generators of $G_{n}^{(i,j)}$. Define a word $w^{i}_{(i,j)}(T)$ in $G_{n}^{(i,j)}$ for $T$ by $w^{i}_{(i,j)}(T) = lk_{c_{1}}lk_{c_{2}} \cdots lk_{c_{m}}$. Denote $\sigma$ by a $n-2$ tuple $(\sigma(k))_{k \in \{1,2,\cdots n \} \backslash \{i,j\}}$. Let $\mathbb{Z}_{2}^{\sigma}$ be the subgroup of $ \mathbb{Z}_{2}^{*2^{n-2}}$ generated by $h(\sigma)$.  
 
 Note that it is easy to get a word $w_{(i,j)}^{i}$ from $T$ because to calculate the word we just count the number of crossings in specific types modulo $\mathbb{Z}_{2}$. Moreover, since $w_{(i,j)}^{i}$ is valued in the free product of $\mathbb{Z}_{2}$, that is, there are no relations except for $a^{2} =1$, it is easy to distinguish two words.

\begin{lem}\label{inv-word-tangle}
 For a positive integer $n$ and for $i, j \in \{1, \cdots ,n\}$ such that $ i \neq j$, $w^{i}_{(i,j)}$ is invariant for oriented enumerated $n-n$ free tangles in a good condition without pure crossings.
\end{lem}

\begin{proof}
Let $T$ and $T'$ be oriented enumerated $n-n$ free tangle diagrams in a good condition without pure crossings. Fix a pair $(i,j)$ such that $i \neq j \in \{1,\cdots,n\}$. Suppose that $T$ and $T'$ are equivalent as free tangles. By Theorem~\ref{Proj-link} and by definition of tangles in a good condition, we may assume that $T'$ is obtained from $T$ by applying one of the second and the third Reidemeister moves. The proof consists of three parts:\\

(1) If two crossings $c$ and $c'$ are contained in the second Reidemeister move, then $lk_{c} = lk_{c'}$.\\

(2) If a crossing $c$ in $T$ is not contained in the second and the third Reidemeister moves, then $lk_{c} = lk_{c'}$ where $c'$ is a crossing in $T'$ associated to $c$ with respect to the move.\\

(3) $w^{i}_{(i,j)}$ is an invariant under the second and the third Reidemeister moves.\\

For two crossings $c$ and $c'$ contained in the second Reidemeister move, since there are no other crossings between $c$ and $c'$, by definition of $lk_{c}$, (1) holds. Consider a crossing $c$ in $T$ which is not contained in the second and the third Reidemeister moves. If we apply the second Reidemeister move, then the number of crossings of type $(i,k)$ and $(j,k)$ from $p_{i}^{s}$ to $c$ and from $p_{j}^{s}$ to $c$ is changed by $0$ or $+2$ or $-2$. Since $lk_{c}$ is defined modulo $\mathbb{Z}_{2}$, $lk_{c} = lk_{c'}$. If we apply the third Reidemeister move, then the number of crossings of type $(i,k)$ and $(j,k)$ from $p_{i}^{s}$ to $c$ from $p_{j}^{s}$ to $c$ is not changed. Hence $lk_{c} = lk_{c'}$ and it is proved that (2) holds.

Now we will show that $w^{i}_{(i,j)}$ is an invariant under the second and the third Reidemeister moves. Clearly, $w^{i}_{(i,j)}$ does not change under the second Reidemeister move because of the relations $\{ \sigma^{2} = 1\}$.
Let us consider the third Reidemeister move and crossings of type $(i,j)$ which are contained in the move. Suppose that the third Reidemeister move consists of $x$-th, $y$-th and $z$-th components. Since $T$ does not consist of pure crossings, $x \neq y \neq z$. By assumption, $\{x,y,z\} \cap \{i,j\} =  \{i,j\}$, say $x=i$ and $y=j$. Then the number of crossings of type $(i,z)$ and $(j,z)$ from $p_{i}^{s}$ to $c$ and from $p_{j}^{s}$ to $c$ is changed by $+2$ or $-2$ and it is not changed modulo $\mathbb{Z}_{2}$. Therefore $lk_{c}$ is not changed for each crossing $c$ of type $(i,j)$ and $w^{i}_{(i,j)}$ is invariant under the second and the third Reidemeister moves.

\end{proof}

Now, let us consider an oriented enumerated free link diagram $L$ $=L_{1} \cup \cdots \cup L_{n}$ in a good condition without pure crossings. For each $i \in \{1,\cdots,n\}$ let us fix a point $p_{i}$ on an edge of $L_{i}$ and $p_{i}^{s}$ and $p_{i}^{e}$ in $\mathbb{R} \times \{0\}$ and $\mathbb{R} \times \{1\}$, respectively. We may assume that $p_{i}^{s} < p_{i+1}^{s}$ and $p_{i}^{e} < p_{i+1}^{e}$ for all $i \in \{1,\cdots,n-1\}$ as elements in $\mathbb{R}$. We cut each component $L_{i}$ at $p_{i}$. Then we get two points $p_{i}^{0}$ and $p_{i}^{1}$ such that the orientaion of $L_{i}$ goes from $p_{i}^{0}$ to $p_{i}^{1}$. 
Let us connect two points $p_{i}^{0}$ and  $p_{i}^{1}$ with $p_{i}^{s}$ and $p_{i}^{e}$, respectively. Herewith, we change new intersections to virtual crossings. 
Then we can get an oriented enumerated $n-n$ free tangle diagram $T_{L}$ in a good condition without pure crossings from $L$, see Fig.~\ref{exa-tangle-link}.

 \begin{figure}[h!]
\begin{center}
 \includegraphics[width = 11cm]{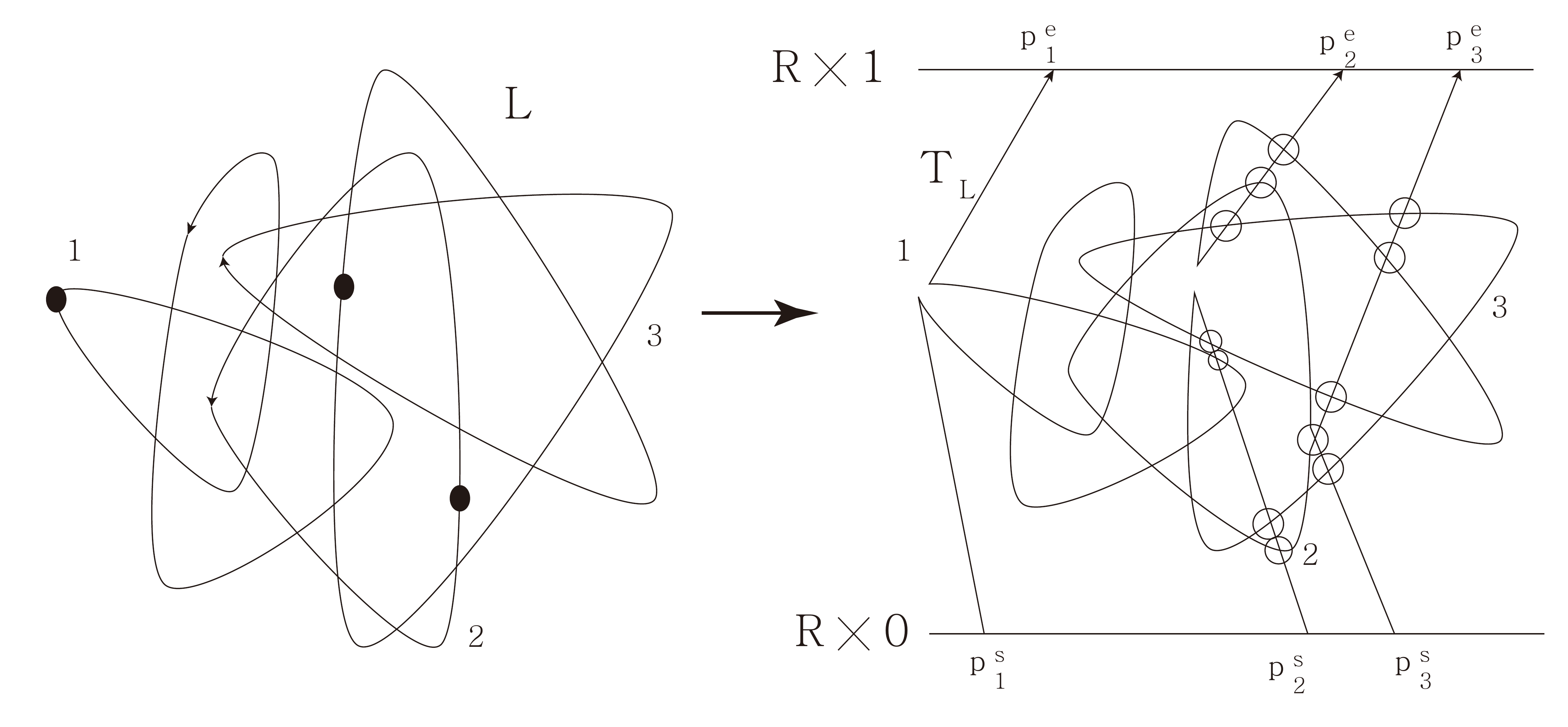}

\end{center}

 \caption{An $n-n$ free tangle from a free link}\label{exa-tangle-link}
\end{figure}

As we already did, for each pair $(i,j)$, we can define $w^{i}_{(i,j)}$. Apparently, the word $w^{i}_{(i,j)}$ does depend on the cut points $\{p_{i}\}_{1}^{n}$. But it does not depend on the points $\{p_{i}^{s},p_{i}^{e}\}$ since every new intersection is considered as a virtual crossing. 

\begin{thm}
Let $L$ $=L_{1} \cup \cdots \cup L_{n}$ be an oriented enumerated free link diagram in a good condition without pure crossings. For each pair $(i,j)$, $w^{i}_{(i,j)}$ is an invariant for oriented enumerated free links in a good condition without pure crossings with respect to slide-equivalence relation in $G_{n}^{(i,j)}$ and conjugation. 
\end{thm}

\begin{proof}
Let $L$ $=L_{1} \cup \cdots \cup L_{n}$ be an oriented enumerated free link diagram in a good condition without pure crossings. Suppose that cut points $\{p_{k}\}_{k=1}^{n}$ are fixed. Let $L'$ be an oriented enumerated free link diagram in a good condition without pure crossings such that $L'$ is obtained from $L$ by a sequence of Reidemeister moves for free diagrams. Then for a pair $(i,j)$, by Theorem~\ref{Proj-link} and Lemma~\ref{inv-word-tangle}, $w^{i}_{(i,j)}(L) = w^{i}_{(i,j)}(L')$. If one of the fixed points $p_{k}$ is moved to another point $p'_{k}$ on $L_{k}$, we may assume that between $p_{k}$ and $p'_{k}$ there is the only one crossing $c$ of type $(k,l)$. Say $\{c_{1}, \cdots, c_{m}\}$ and $\{c'_{1}, \cdots, c'_{m}\}$  are sets of crossings of type $(i,j)$ on $L_{i}$ such that crossings in $\{c_{1}, \cdots, c_{m}\}$ (in $\{c'_{1}, \cdots, c'_{m}\}$) are ordered with respect to the orientation of $L_{i}$ and to $p_{k}$ (to $p'_{k}$). If $k \neq i$ and $k \neq j$, then $w^{i}_{(i,j)}$ is not changed. Suppose that $k =i$. If $l \neq j$, then $c_{a} = c'_{a}$ for all $a  \in  \{1, \cdots, m \}$ and $lk_{c_{a}}(x)$ is changed from $0$ to $1$(or $1$ to $0$), that is, $lk_{c'_{a}} = f_{x}(lk_{c_{a}})$. Therefore $w^{i}_{(i,j)}$and $w'^{i}_{(i,j)}$ are slide-equivalent in $G_{n}^{(i,j)}$.
Suppose that $l = j$ and the orientation of $L_{i}$ go from $p_{i}$ to $p'_{i}$. Note that the crossing $c = c_{1}$ with respect to $p_{i}$ and $c'_{a} = c_{a+1}$ for $a \in \{1, \cdots m-1\}$ and $c'_{m} = c_{1}$. Since $lk_{c'_{a}} = lk_{c_{a+1}}$ for $a \in \{1, \cdots m-1\}$ and $lk_{c'_{m}} = lk_{c_{1}}$, the word $w'^{i}_{(i,j)}$ is $lk^{-1}_{c}w^{i}_{(i,j)}lk_{c}$.

\end{proof}

\begin{exa}

Let us calculate the word $w^{1}_{(1,2)}$. There are four crossings of type $(1,2)$ and $G_{n}^{(1,2)} \cong \mathbb{Z}_{2}^{(0,0)} * \mathbb{Z}_{2}^{(0,1)} * \mathbb{Z}_{2}^{(1,0)} * \mathbb{Z}_{2}^{(1,1)}$. Let the generators of  $\mathbb{Z}_{2}^{(0,0)}$, $\mathbb{Z}_{2}^{(0,1)}$, $\mathbb{Z}_{2}^{(1,0)}$ and $\mathbb{Z}_{2}^{(1,1)}$ be $\alpha$, $\beta$, $\gamma$ and $\delta$, respectively. With respect to the orientation of the first component, give an order of those crossings $c_{1},$ $c_{2},$ $c_{3},$ and $c_{4}$. By some calculations, we get $lk_{c_{1}} = (0,0)$, $lk_{c_{2}} = (0,1)$, $lk_{c_{3}} = (1,1)$, $lk_{c_{4}} = (0,0)$. Therefore $w^{1}_{(1,2)} = \alpha\beta\delta\alpha = \beta\delta$ modulo conjugations. Since the word cannot be reduced to the trivial, this link is not trivial. 
 \begin{figure}[h!]
\begin{center}
 \includegraphics[width = 9cm]{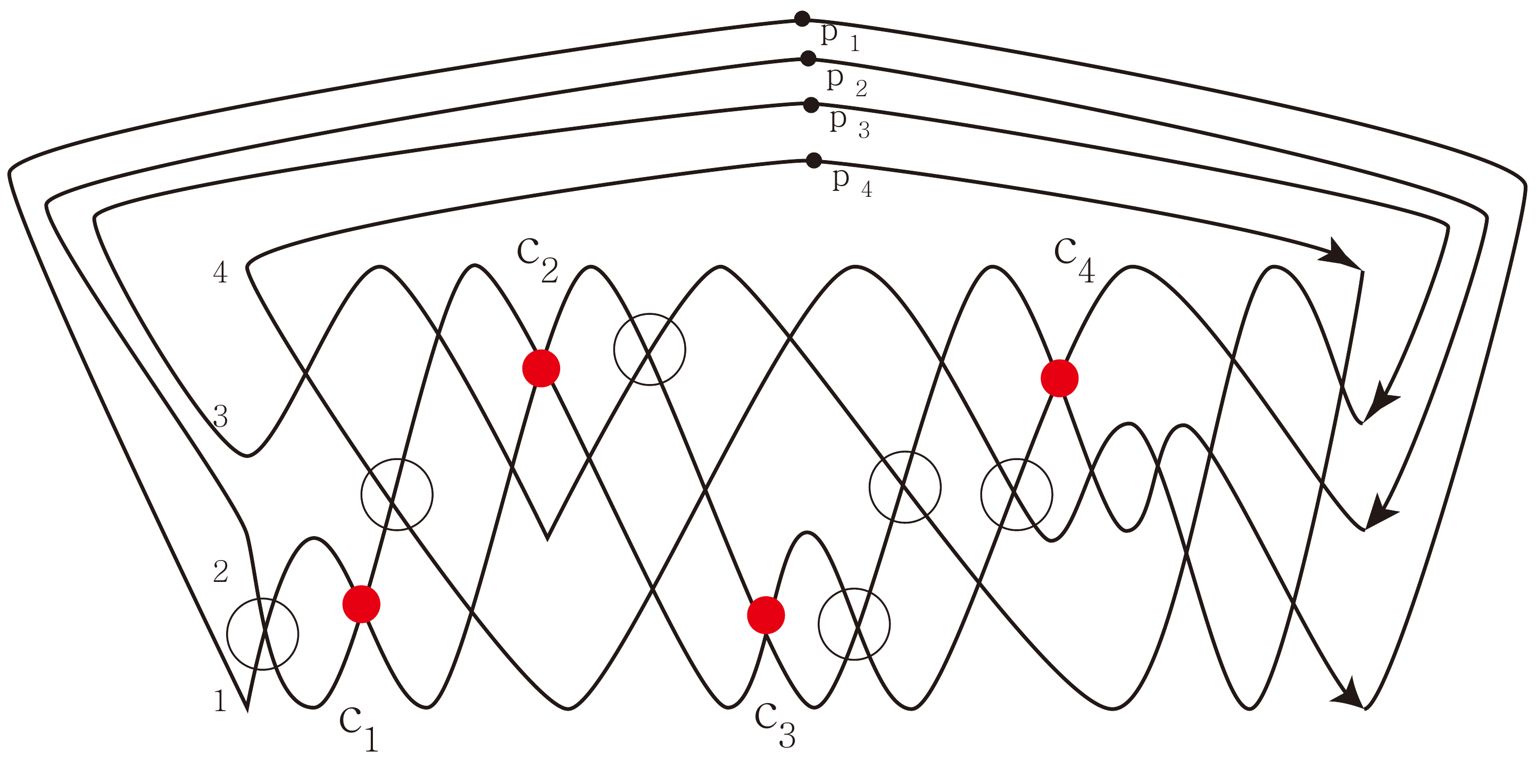}

\end{center}

 \caption{}\label{exa-inv1}
\end{figure}

\end{exa}

\end{document}